\documentclass[preprint,12pt]{elsarticle}
\usepackage{mathrsfs}
\usepackage{}
\usepackage{amsfonts}
\usepackage{amssymb}
\usepackage{amsthm}
\usepackage{amsmath}
\usepackage{amssymb}
\usepackage{amsthm}
\usepackage{color}
\usepackage{longtable}
\newtheorem{theorem}{Theorem}
\newtheorem{lemma}{Lemma}

\newdefinition{example}{Example}
\newdefinition{open}{Open question}

\newdefinition{definition}{Definition}
\allowdisplaybreaks[2]

\journal{arXiv}

\begin{document}

\begin{frontmatter}

\title{The sequence reconstruction problem for permutations with the Hamming distance}

\author[1]{Xiang Wang}
\author[2,3,4]{Elena~V.~Konstantinova$^*$}

\address[1]{Faculty of Science, Beijing University of Technology, Beijing, 100124, China, \\E-mail address: xwang@bjut.edu.cn}
\address[2]{Sobolev Institute of Mathematics, Ak. Koptyug av. 4, Novosibirsk 630090, Russia}
\address[3]{Novosibirsk State University, Pirogova str. 2, Novosibirsk, 630090, Russia}
\address[4] {Three Gorges Mathematical Research Center, China Three Gorges University, 8 University Avenue, Yichang 443002, Hubei Province, China \\E-mail address: e\_konsta@math.nsc.ru}

\cortext[cor2]{Corresponding author: Elena~V.~Konstantinova}
\begin{abstract}
V.~Levenshtein first proposed the sequence reconstruction problem in 2001. This problem studies the model where the same sequence from some set is transmitted over multiple channels, and the decoder receives the different outputs. Assume that the transmitted sequence is at distance $d$ from some code and there are at most $r$ errors in every channel. Then the sequence reconstruction problem is to find the minimum number of channels required to recover exactly the transmitted sequence that has to be greater than the maximum intersection between two metric balls of radius $r$, where the distance between their centers is at least $d$. In this paper, we study the sequence reconstruction problem of permutations under the Hamming distance. In this model we define a Cayley graph over the symmetric group, study its properties and find the exact value of the largest intersection of its two metric balls for $d=2r$. Moreover, we give a lower bound on the largest intersection of two metric balls for $d=2r-1$.
\end{abstract}


\end{frontmatter}

\section{Introduction}
\label{intro}

The sequence reconstruction problem was proposed in coding theory by V.~Levenshtein~\cite{L1} in $2001$ as a local reconstruction of sequences in the model where the same sequence is transmitted over multiple channels, and the decoder receives all the distinct outputs. This problem is stated as follows. 

Let $S$ be a set of all sequences of length $n$, $\rho$ be a metric in $S$, and $B_r(x)=\{y\in S|\rho(x,y)\leqslant r\}$ be a metric ball of radius $r$ centered at $x\in S$. For any integer $d\geqslant 1$, the minimum number of transmission channels has to be greater than the maximum intersection of two metric balls centered at elements of $S$ denoted as follows:
\begin{equation}
N(n,d,r)=\max\limits_{x_1,x_2\in S, \rho(x_1,x_2)\geqslant  d}|B_r(x_1)\cap B_r(x_2)|.\label{eq1}
\end{equation}
The problem of determining $N(n,d,r)$ is the \emph{sequence reconstruction problem}.

The model above is presented by an error graph where vertices correspond to sequences and edges connect vertices under error transmissions. From a graph-theoretical point of view, this is the problem of reconstructing a vertex by its neighbours being at a given distance from the vertex~\cite{K4}. In~\cite{L1,L2} this problem was completely solved for the Hamming graph and the Johnson graph. The main advantage in getting exact values of the minimum number of transmission channels required to recover the transmitted sequence exactly is that both graphs are distance-regular. This structural property allows to apply general approaches to solve the problem for any vertex in a graph.

The situation is completely changed whenever a graph is not a distance-regular. In~\cite{K4} the main difficulties in solving this problem for Cayley graphs over the symmetric group and the hyperoctahedral group that are not distance-regular graphs are discussed. In particular, this problem was studied in~\cite{K1,K3} over permutations with reversal errors and was completely solved for $r=1$. However, even for $r=2$ there are only lower bounds on~(\ref{eq1}) obtained. To get these bounds, all possible cases of neighbourhoods of a vertex are considered. The same difficulties are appeared in~\cite{K2,L4} when the reconstruction problem is solved for the transposition graph. To get a complete solution, the reconstruction of elements in the symmetric group is considered with paying attention to conjugacy classes to which group elements belong~\cite{L4}.

A particular case of transposition errors, namely, transpositions of two adjacent elements of a permutation lead to a Cayley graph over the symmetric group known as the bubble-sort graph~(see~\cite{K4} for details). The distance in this graph is called the bubble-sort distance in computer science~\cite{Knu94} or  Kendall’s $\tau$-metric in statistics~\cite{KG90}. Despite this graph is an induced subgraph of the transposition graph, the complete solution of the sequence reconstruction problem of permutations distorted by these errors is still unknown. Since the bubble-sort graph is not distance-transitive this cause difficulties in finding~(\ref{eq1}) in a general case. Some particular cases were solved in~\cite{Yaakobi,Wang1}. 

In this paper, we study the sequence reconstruction problem of permutations by the single Hamming errors. First, we present some properties of $|B_r(\pi)\cap B_r(\tau)|$ over the symmetric group with the Hamming distance for any permutations $\pi$ and $\tau$. Then we define the Cayley graph $\mathrm{Sym}_n(H), n\geqslant 2,$ over the symmetric group $\mathrm{Sym}_n$ generated by cycles of length at least two and show that the distance between two permutations in this graph is the Hamming distance. Since the graph  $\mathrm{Sym}_n(H)$ is a vertex-transitive graph, we study $N(n,d,r)=\max_{\pi\in\mathrm{Sym}_n\\\pi\neq I_n}|B_r(I_n)\cap B_r(\pi)|,$ where $I_n$ is the identity permutation, the distance between $\pi$ and $I_n$ is at least $d$, and $n\geqslant r\geqslant \frac{d}{2}$. 

The rest of this paper is organised as follows. In Section~\ref{sec2}, we formally give definitions and notation with respect to the sequence reconstruction problem over permutations under the Hamming distance. In Section~\ref{sec3}, the properties of the Cayley graph $\mathrm{Sym}_n(H), n\geqslant 2,$ are studied. In particular, it is shown that this graph is not distance regular. In Sections~\ref{sec4} and~\ref{sec5} we present the key technical lemma and obtain $N(n,d,r)$ for $d=2r$.  In Section~\ref{sec6} we give the lower bound on the value of $N(n,d,r)$
for $d=2r-1$.

\section{Preliminaries}
\label{sec2}

Let $\mathrm{Sym}_n, n\geqslant 2,$ be the symmetric group of permutations $\pi=[\pi_1 \pi_2 \ldots \pi_n]$ written as strings in one-line notation, where $\pi_i=\pi(i)$ for any $1\leqslant i \leqslant n$, with the identity element $I_n=[1\,2\,\ldots\,n]$. It is well-known fact that any permutation can be expressed as a product of disjoint cycles. For any $\pi\in \mathrm{Sym}_n$, let $disc(\pi) = [1^{h_1} 2^{h_2}\ldots n^{h_n}]$ be the cycle type of $\pi$, where $h_i$ is the number of cycles of length $i$. We omit components with $h_i=0$ in $disc(\pi)$. For example, the cycle type of $\pi=(1\,2)(3\,4\,5)(6\,7\,8)\in\mathrm{Sym}_8$ is written as $disc(\pi)=[2^13^2]$.

For any two permutations $\pi$ and $\tau$, the Hamming distance between them is the number of positions in which these permutations differ:
\begin{equation}
d(\pi,\tau)=|\{i\in[n]|\pi_i\neq \tau_i\}|,\label{eq2}
\end{equation}
where $[n]=\{1,2,...,n-1,n\}$.

Let $B_r(\pi)=\{\tau\in \mathrm{Sym}_n|d(\pi,\tau)\leqslant r\}$ and $S_r(\pi)=\{\tau\in \mathrm{Sym}_n|d(\pi,\tau)=r\}$ be a metric ball and a metric sphere of radius $r$ centered at a permutation $\pi$. The sizes of $B_r(\pi)$ and $S_r(\pi)$ do not depend on a permutation $\pi$ under the Hamming distance~\cite{CKLL97}. For convenience, we put $B_r(n)=|B_r(\pi)|$ and $S_r(n)=|S_r(\pi)|$ for any $\pi\in \mathrm{Sym}_n$.

Now we give some useful definitions and notations on enumerative combinatorics following by R.~Stanley~\cite{Stanley11}. A {\it derangement} of order $r$ is a permutation $\pi$ with no fixed points, i.e., $\pi_i\neq i$ for any $i\in [r]$. The number $D_r$ of distinct derangements on $r$ elements is given by the following formula:
\begin{equation}
D_r=r! \sum_{i=0}^r \frac{(-1)^i}{i!}. \nonumber 
\end{equation}
Then the size of $S_r(n)$ is given as follows:
\begin{equation}
S_r(n)=\binom{n}{r}D_r,\nonumber
\end{equation}
and the size of $B_r(n)$ is presented as follows:
\begin{equation}
B_r(n)=1+\sum\limits_{i=1}^{r}S_i(n)=1+\sum\limits_{i=1}^{r}\binom{n}{i}D_i.\nonumber
\end{equation}

For two integers $d$ and $r$, let $I(n,d,r)$ be the size of the maximum intersection of two metric balls of radius $r$ and at distance $d$ between their centers $\pi,\tau\in \mathrm{Sym}_n$ such that:
\begin{equation}
I(n,d,r)=\max\limits_{\pi,\tau\in \mathrm{Sym}_n, d(\pi,\tau)=d}|B_r(\pi)\cap B_r(\tau)|.\nonumber
\end{equation}
The formula~(\ref{eq1}) can be rewritten in terms of permutations as follows:
\begin{equation}
N(n,d,r)=\max\limits_{\pi,\tau\in \mathrm{Sym}_n, d(\pi,\tau)\geqslant d}|B_r(\pi)\cap B_r(\tau)|=\max\limits_{k\geqslant d}I(n,k,r).\label{eq3}
\end{equation}

We say that $C\subset \mathrm{Sym}_n$ is a $(n,d)$-permutation code under the Hamming distance for $1\leqslant d\leqslant n$, if $d(\pi,\tau)\geqslant d$ for any distinct permutations $\pi,\tau\in C$. Assume a permutation $\pi\in C$ is transmitted over $N$ channels, where $C$ is an $(n,d)$-permutation code, and there are at most $r$ errors on each channel such that all channel outputs differ from each other. V.~Levenshtein~\cite{L1} proved that the minimum number of channels that guarantees the decoder decodes successfully any transmitted codeword from $C$ is given by $N(n,d,r)+1$.

\section{The Cayley graph over the symmetric group with the Hamming distance}
\label{sec3}
In this section, we define a Cayley graph over the symmetric group of permutations with the Hamming distance and give some properties of sequence reconstruction in this Cayley graph.

Let $\mathrm{Sym}_n(H), n\geqslant 2,$ be the Cayley graph over the symmetric group $\mathrm{Sym}_n$ generated by cycles of length at least two from the following set:
\begin{equation}
H=\{\gamma\in\mathrm{Sym}_n|\, disc(\gamma)=[1^{n-i}i^1], i\in[n]\backslash\{1\}\}.\label{eq4}
\end{equation}
If $\gamma\in H$ with $disc(\gamma)=[1^{n-i}i^1]$ then it is obvious that $disc(\gamma^{-1})=[1^{n-i}i^1]$, and hence $H=H^{-1}$. For any $\pi\in \mathrm{Sym}_n$ and any $\gamma\in H$ with $disc(\gamma)=[1^{n-i}i^1]$ and $disc(\gamma^{-1})=[1^{n-i}i^1]$, the weight of an edge $\{\pi,\pi\,\gamma\}$ in $\mathrm{Sym}_n(H)$ is defined as $i$, where $\pi\,\gamma(j)=\gamma(\pi(j))$ for any $j\in[n]$. The distance $d(\pi,\tau)$ between two permutations $\pi$ and $\tau$ in this graph is defined as the length of the shortest path between $\pi$ and $\tau$, that is the least sum of lengths of disjoint cycles transforming $\pi$ into $\tau$. By~(\ref{eq2}), we have the following statement.

\begin{lemma}
The distance between two permutations in $\mathrm{Sym}_n(H), n\geqslant 2,$ is the Hamming distance.
\label{lm1}
\end{lemma}
\begin{proof}
For any two permutations $\pi,\tau$ of the graph $\mathrm{Sym}_n(H)$, the distance between $\pi$ and $\tau$ is the least sum of lengths of disjoint cycles transforming $\pi$ into $\tau$, where each cycle contributes a value with the length of the cycle to the Hamming distance between $\pi$ and $\tau$. Therefore, the lemma follows.
\end{proof}

The number of permutations of a given cycle type is as follows.
\begin{lemma}\cite[Proposition 1.3.2]{Stanley11}
\label{lm2}
The number of permutations $\pi\in\mathrm{Sym}_n$ of cycle type $[1^{n-i}i^1]$ is equal to $\frac{n!}{(n-i)!\cdot i}$.
\end{lemma}

Thus, by Lemma $\ref{lm2}$, it follows that
\begin{equation}
|\{\pi|\pi\in\mathrm{Sym}_n,disc(\pi)=[1^{n-i}i^1]\}|=\frac{n!}{(n-i)!\cdot i},\nonumber
\end{equation}
for any $i\in [n]\backslash\{1\}$. Moreover, there exists an edge of weight $d(\pi,\pi\,\gamma)=i$ for any  $\pi\in\mathrm{Sym}_n$ and any $\gamma\in H$ with $disc(\gamma)=[1^{n-i}i^1]$, $2\leqslant i\leqslant n$. By~(\ref{eq4}), we immediately have: 
$$|H|=\sum\limits_{i=2}^{n}\frac{n!}{(n-i)!\cdot i}.$$

\begin{example} It is obvious that $\mathrm{Sym}_2(H)\cong K_2$ and $\mathrm{Sym}_3(H)\cong K_6$. The graph $\mathrm{Sym}_4(H)$ has the following generating set: 
$$H=\{(1\,2),(1\,3),(1\,4),(2\,3),(2\,4),(3\,4),(1\,2\,3),(1\,3\,2),(1\,2\,4),(1\,4\,2),(1\,3\,4),$$ $$(1\,4\,3),(2\,3\,4),(2\,4\,3),(1\,2\,3\,4),(1\,2\,4\,3),(1\,3\,2\,4),(1\,3\,4\,2),(1\,4\,2\,3),(1\,4\,3\,2)\}$$ of size $|H|=20$. Moreover, each vertex of the graph is incident to six edges of weight two, eight edges of weight three, and six edges of weight four. However, it is not a distance-regular graph since, as example,  for permutations $\pi=(1\,2\,3\,4)$ and $\tau=(1\,2)(3\,4)$ that are both at distance four from the identity permutation $I_4$, the number of permutations $x$ at distance $d(x,I_4)=d(x,\pi)=2$ is equal to zero while the number of permutations $y$ at distance $d(y,I_4)=d(y,\tau)=2$ is equal to two.
\end{example}

Actually, it is obvious that more general result holds.

\begin{lemma}
The Cayley graph $\mathrm{Sym}_n(H), n\geqslant 4,$ is not distance-regular.
\label{lm3}
\end{lemma}

Now without loss of generality we consider a distance between the identity permutation $I_n$ and a permutation $\pi$ in the graph $\mathrm{Sym}_n(H)$. 

Let $w_H(\pi)=d(I_n,\pi)$ be the Hamming weight of $\pi$.

\begin{lemma}
\label{lm4}
For any permutation $\pi\in \mathrm{Sym}_n$ with $disc(\pi)= [1^{h_1} 2^{h_2}\ldots n^{h_n}]$, we have:
\begin{equation}
w_H(\pi)=\sum\limits_{i=2}^{n}ih_i. \nonumber
\end{equation}
Moreover, $\sum\limits_{i=1}^{n}ih_i=n$.
\end{lemma}
\begin{proof}
If $\pi=(a_1\ldots a_m)$ is a cycle of order $m$ then $d(I_n,\pi)=m$. If $\pi=(a_1\ldots a_s) (b_1\ldots b_t)$, where $(a_1\ldots a_s)$ and $(b_1\ldots b_t)$ are disjoint cycles, then  $\pi(i)\neq i$ for any $i\in \{a_1,\ldots,a_s,b_1,\ldots,b_t\}$. Hence, we have $d(I_n,\pi)=t+s$. If $disc(\pi)=[1^{h_1}\ldots n^{h_n}]$ then $w_H(\pi)=d(I_n,\pi)=\sum\limits_{i=2}^{n}ih_i$, and $\sum\limits_{i=1}^{n}ih_i=n.$
\end{proof}

\section{The key technical lemma}
\label{sec4}
Let $t$ and $k$ be integers. For convenience, we put $\binom{t}{k}=0$ for each $k<0$.  
\begin{lemma}
\label{lm5}
For $1\leqslant k \leqslant  \frac{2t+1}{3}$,
\begin{equation} \label{eq5}
\sum\limits_{j=0}^{k-1}\binom{2t+1-3k}{t-1-3j}\binom{2k}{1+2j}\leqslant 2\binom{2t-2}{t-1}.
\end{equation}
\end{lemma}

\begin{proof} If $k=1,2,3$ then the left part of~(\ref{eq5}) is equal to $$2\binom{2t-2}{t-1}, \ \  8\binom{2t-5}{t-1}, \ \  12\binom{2t-8}{t-1}+20\binom{2t-8}{t-4},$$ respectively.

If $k\geqslant 4$ then we have:
\begin{small}
\begin{align}
\sum\limits_{j=0}^{k-1}\binom{2t+1-3k}{t-1-3j}&\binom{2k}{1+2j}=2k\binom{2t+1-3k}{t-1}+\binom{2k}{3}\binom{2t+1-3k}{t-4}\nonumber\\
&+\binom{2k}{5}\binom{2t+1-3k}{t-7}+\sum\limits_{j=3}^{k-1}\binom{2t+1-3k}{t-1-3j}\binom{2k}{1+2j}.\nonumber
\end{align}
\end{small}
Then, it follows that
\begin{small}
\begin{align}
\sum\limits_{j=0}^{k-1}&\binom{2t+1-3k}{t-1-3j}\binom{2k}{1+2j}\nonumber\\
&\overset{(a)}{\leqslant} 2k\binom{2t+1-3k}{t-1}+\sum\limits_{j=1}^{k-1}\binom{2t+1-k}{t-j}\nonumber\\
&\overset{(b)}{\leqslant} 2k\binom{2t+1-3k}{t-1}+\binom{2t+1-k}{t-1}+\binom{2t+1-k}{t-2}\nonumber\\
&~~~~~~~~~~~~~~~~~~~~~~~~~~~~~~~~~~~~~~~~~~~~+\sum\limits_{j=3}^{k-2}\binom{2t-1-j}{t-j}+\binom{2t+1-k}{t-k+1}\nonumber\\
&\overset{(c)}{=}2k\binom{2t+1-3k}{t-1}+\binom{2t+1-k}{t-1}+\binom{2t+1-k}{t-2}+\binom{2t-3}{t-3},
\label{eq6}
\end{align}
\end{small}
where $\overset{(a)}{\leqslant}$ follows from $$\binom{2t+1-3k}{t-1-3j}\binom{2k}{1+2j}\leqslant \binom{2t+1-k}{t-j}$$ for every $1\leqslant j\leqslant k-1$, $\overset{(b)}{\leqslant}$ follows from $$\binom{2t+1-k}{t-j}\leqslant \binom{2t-1-j}{t-j}$$ for every $3\leqslant j\leqslant k-2$, and $\overset{(c)}{=}$ follows from $$\binom{t}{k}=\binom{t-1}{k}+\binom{t-1}{k-1}$$ for every $0\leqslant k\leqslant t$.

Now we compare the above values. Since 
$$\frac{\binom{2t-2}{t-1}}{\binom{2t-5}{t-1}}=\frac{(2t-4)}{(t-3)}\cdot \frac{(2t-3)(2t-2)}{(t-2)(t-1)}\geqslant 8,$$ 
then 
$$\binom{2t-2}{t-1}\geqslant 4\binom{2t-5}{t-1}.$$ 

Moreover, it can be shown that:

$$8\binom{2t-5}{t-1}\geqslant 12\binom{2t-8}{t-1}+20\binom{2t-8}{t-4}.$$

For $k\geqslant 5$, the following holds:
$$\frac{\binom{2t-3}{t-1}}{2k\binom{2t+1-3k}{t-1}}=\frac{(2t-3)!(t-1)!(t+2-3k)!}{2k(t-1)!(t-2)!(2t+1-3k)!}=\frac{\prod_{j=1}^{3k-4}(2t+1-3k+j)}{2k \prod_{j=1}^{3k-4}(t+2-3k+j)}\geqslant$$ $$\frac{1}{2k}\left(\frac{2t-3}{t-2}\right)^{3k-4}\geqslant \frac{2^{3k-4}}{2k}\geqslant 1.$$ 

Since for $k\geqslant 5$ we have $$\binom{2t-3}{t-1}\geqslant \binom{2t+1-k}{t-1},$$ $$\binom{2t-3}{t-2}\geqslant \binom{2t+1-k}{t-2},$$ and $$2\binom{2t-3}{t-2}\geqslant \binom{2t-3}{t-3}$$ then it follows that for $1\leqslant k \leqslant  \frac{2t+1}{3}$ we have: 
\begin{small}
\begin{align}
2\binom{2t-2}{t-1}&=2\binom{2t-3}{t-1}+2\binom{2t-3}{t-2}\nonumber\\
&\geqslant 2k\binom{2t+1-3k}{t-1}+\binom{2t+1-k}{t-1}+\binom{2t+1-k}{t-2}+\binom{2t-3}{t-3}\nonumber\\
&\overset{(d)}{\geqslant}\sum\limits_{j=0}^{k-1}\binom{2t+1-3k}{t-1-3j}\binom{2k}{1+2j},\nonumber
\end{align}
\end{small}
where $\overset{(d)}{\geqslant}$ follows from~(\ref{eq6}). This gives~(\ref{eq5}) and completes the proof.
\end{proof}

\section{The values of $N(n,d,r)$ for $d=2r$}
\label{sec5}

Let $\pi\in \mathrm{Sym}_n$ be a product of disjoint cycles which we call cycles of $\pi$. We denote $Ts(\pi)=\{i\in[n]|\pi(i)\neq i\}$ and $Tc(\pi)=\{<i,\pi(i)>|\pi(i)\neq i~\text{for}~i\in[n]\}$. For example, if $\pi=(1\,2)(4\,5\,6)$ then $Ts(\pi)=\{1,2,4,5,6\}$, $Tc(\pi)=\{<1,2>,<2,1>,<4,5>,<5,6>,<6,4>\}$. 

\begin{lemma}
\label{lm6}
For any two permutations $\pi,\tau\in \mathrm{Sym}_n$, the following equality holds:
\begin{equation}
d(\pi,\tau)=|Ts(\pi)\cup Ts(\tau)|-|Tc(\pi)\cap Tc(\tau)|.\nonumber
\end{equation}
\end{lemma}
\begin{proof}
For any $\pi(i)\neq \tau(i)$, by the definition of $Ts(\pi)$ and $Ts(\tau)$, we have $i\in Ts(\pi)\cup Ts(\tau)$. Moreover, if $i\in Ts(\pi)\cup Ts(\tau)$ and $\pi(i)=\tau(i)=j$ then $<i,j>\in Tc(\pi)\cap Tc(\tau)$. If $<i,j>\in Tc(\pi)\cap Tc(\tau)$ then $\pi(i)=\tau(i)=j\neq i$ and $i\in Ts(\pi)\cup Ts(\tau)$. Since $d(\pi,\tau)=|\{i\in [n]|\pi(i)\neq \tau(i)\}|$ we have $d(\pi,\tau)=|Ts(\pi)\cup Ts(\tau)|-|Tc(\pi)\cap Tc(\tau)|$.
\end{proof}

\begin{lemma}
\label{lm7}
For any two permutations $\pi,\tau\in \mathrm{Sym}_n$, if $disc(\pi)=disc(\tau)$ then we have:
\begin{equation}
d(I_n,\pi)=d(I_n,\tau)~\text{and}~
|B_r(\pi)\cap B_r(I_n)|=|B_r(\tau)\cap B_r(I_n)|\nonumber
\end{equation}
for any $n\geqslant r$.
\end{lemma}
\begin{proof}
By Lemma~$\ref{lm4}$, if $disc(\pi)=disc(\tau)$ then $d(I_n,\pi)=d(I_n,\tau)$. 
Let $\pi=(a_1\ldots a_m)$ and $\tau=(b_1\ldots b_m)$. Consider a permutation $f:[n]\rightarrow[n]$ acting on elements of cycles such that $f(\sigma)=(f(c_{1}^{1})\,\ldots\,f(c_{l_1}^{1}))\ldots(f(c_{1}^{s})\,\ldots\,f(c_{l_s}^{s}))$, $\sum\limits_{i=1}^{s} l_i=n$, for each permutation $\sigma=(c_{1}^{1}\,\ldots\,c_{l_1}^{1})\ldots(c_{1}^{s}\,\ldots\,c_{l_s}^{s})$ expressed as a product of disjoint cycles. Clearly, $f(\sigma)$ is the product of disjoint cycles and $disc(\sigma)=disc(f(\sigma))$ for each $\sigma\in \mathrm{Sym}_n$.

Given any two permutations $\pi,\tau\in \mathrm{Sym}_n$ with $disc(\pi)=disc(\tau)$, we have that $\pi=(c_{1}^{1}\,\ldots\,c_{l_1}^{1})\ldots(c_{1}^{s}\,\ldots\,c_{l_s}^{s})$ and $\tau=(d_{1}^{1}\,\ldots\,d_{l_1}^{1})\ldots(d_{1}^{s}\,\ldots\,d_{l_s}^{s})$, where $\sum\limits_{i=1}^{s} l_i=n$. Here, we define $f(c_{i}^{j})=d_{i}^{j}$ for all $1\leqslant i\leqslant l_j$ and $1\leqslant j\leqslant s$. It follows that $f$ is a permutation and $f(\pi)=\tau$.

If $\sigma\in B_r(\pi)\cap B_r(I_n)$ then $w_H(\sigma)\leqslant r$ and $d(\sigma,\pi)\leqslant r$. 
By Lemma~$\ref{lm4}$, we have $w_H(f(\sigma))=w_H(\sigma)\leqslant r$, and by Lemma~$\ref{lm6}$ we have:
$$d(\sigma,\pi)=|Ts(\sigma)\cup Ts(\pi)|-|Tc(\sigma)\cap Tc(\pi)|$$ 
and 
$$d(f(\sigma),f(\pi))=|Ts(f(\sigma))\cup Ts(f(\pi))|-|Tc(f(\sigma))\cap Tc(f(\pi))|.$$ 

Since $f$ is a permutation we have $|Ts(f(\sigma))\cup Ts(f(\pi))|=|Ts(\sigma)\cup Ts(\pi)|$. If $<i,j>\in Tc(\sigma)\cap Tc(\pi)$ then $\sigma(i)=j$ and $\pi(i)=j$. Thus, $i$ and $j$ are adjacent in some cycle of $\sigma$ and in $\pi$, and hence $f(i)$ and $f(j)$ are adjacent in some cycle of $f(\sigma)$ and in $f(\pi)$. Therefore, we have: 
$$<f(i),f(j)>\in Tc(f(\sigma))\cap Tc(f(\pi)).$$ 

So, $d(f(\sigma),f(\pi))=d(f(\sigma),\tau)\leqslant d(\sigma,\pi)=r$ and $f(\sigma)\in B_r(\tau)\cap B_r(I_n)$. Hence $|B_r(\pi)\cap B_r(I_n)|\leqslant |B_r(\tau)\cap B_r(I_n)|$. Similarly, there exists a permutation $f'(\tau)=\pi$ such that $f'(\sigma)\in B_r(\pi)\cap B_r(I_n)$ for each $\sigma\in B_r(\tau)\cap B_r(I_n)$, then $|B_r(\pi)\cap B_r(I_n)|\geqslant |B_r(\tau)\cap B_r(I_n)|$. Therefore, we have $|B_r(\pi)\cap B_r(I_n)|=|B_r(\tau)\cap B_r(I_n)|$ which complete the proof.
\end{proof}

Since the graph $\mathrm{Sym}_n(H)$ is vertex-transitive, in what it follows below we study 
$$I(n,d,r)=\max_{\pi,I_n\in\mathrm{Sym}_n, d(I_n,\pi)=d}|B_r(I_n)\cap B_r(\pi)|$$ instead of considering $N(n,d,r)$, where $n\geqslant r\geqslant \frac{d}{2}$.

\begin{lemma}
\label{lm8}
For any $n\geqslant 2r$, there exists a permutation $\pi \in \mathrm{Sym}_n$ with $w_H(\pi)=2r$ such that:
\begin{equation}
I(n,2r,r)=|B_r(\pi)\cap B_r(I_n)|=\sum\limits_{i=0}^{r}|\{\sigma|w_H(\sigma)=i,d(\sigma,\pi)\leqslant r\}|.\nonumber
\end{equation}
Moreover, if $\sigma\in B_r(\pi)\cap B_r(I_n)$ then $Tc(\sigma)\subset Tc(\pi)$ and $|Tc(\sigma)|=r$.
\end{lemma}
\begin{proof}
By Lemmas~$\ref{lm1}$ and $\ref{lm7}$, there exists a permutation $\pi$ with $w_H(\pi)=2r$ such that:
\begin{equation}
I(n,2r,r)=|B_r(\pi)\cap B_r(I_n)|=\sum\limits_{i=0}^{r}|\{\sigma|w_H(\sigma)=i,d(\sigma,\pi)\leqslant r\}|.\nonumber
\end{equation}
Moreover, we have $d(\sigma,\pi)=|Ts(\sigma)\cup Ts(\pi)|-|Tc(\sigma)\cap Tc(\pi)|\geqslant 2r-w_H(\sigma)$. Hence, for any $i$ such that $0\leqslant i\leqslant r-1$ we have: $$|\{\sigma|w_H(\sigma)=i,d(\sigma,\pi)\leqslant r\}|=0.$$ The case when $w_H(\sigma)=r$ appears if $|Ts(\sigma)\cup Ts(\pi)|=2r$ and $|Tc(\sigma)\cap Tc(\pi)|=r$. Hence, in this case we have $d(\sigma,\pi)=r$, and furthermore, $Tc(\sigma)\subset Tc(\pi)$ with $|Tc(\sigma)|=r$ and $|Tc(\pi)|=2r$. Therefore, for each cycle $(x_1\,\ldots\,x_l)$ of $\sigma$ there should be a cycle of $\pi$, i.e., $\sigma$ is presented by a product of disjoint cycles which are cycles of $\pi$.
\end{proof}

\begin{lemma}
\label{lm9}
For any $n\geqslant 2r$, there exists a permutation $\pi \in \mathrm{Sym}_n$ whose cycles are only of lengths two or three and $$I(n,2r,r)=|B_r(\pi)\cap B_r(I_n)|.$$
\end{lemma}
\begin{proof}
Let $\pi'\in \mathrm{Sym}_n$ with $I(n,2r,r)=|B_r(\pi')\cap B_r(I_n)|$ and $w_H(\pi')=2r$. Assume that $\pi'$ has a cycle $(y_1\,y_2\,\ldots\,y_t)$ of length $t\geqslant 4$. We consider a permutation $\pi$ that has all the cycles of $\pi'$ but the cycle $(y_1\,\ldots\,y_t)$ along with two additional cycles $(y_1\,y_2)$ and $(y_3\,\ldots\,y_t)$.

Let $\sigma \in B_r(\pi')\cap B_r(I_n)$. Then $Tc(\sigma)\subset Tc(\pi')$ with $|Tc(\sigma)|=r$. If $\sigma$ does not contain  $(y_1\,\ldots\,y_t)$ then $Tc(\sigma)\subset Tc(\pi)$. Otherwise, there exists a permutation $\sigma'$ that has all the cycles of $\sigma$ but $(y_1\,\ldots\,y_t)$ with two additional cycles $(y_1\,y_2)$ and $(y_3\,\ldots\,y_t)$, $|Tc(\sigma')|=r$, and $Tc(\sigma')\subset Tc(\pi)$. Thus, $|B_r(\pi)\cap B_r(I_n)|\geqslant |B_r(\pi')\cap B_r(I_n)|=I(n,2r,r)$. Therefore, there exists a permutation $\pi$ such that $I(n,2r,r)=|B_r(\pi)\cap B_r(I_n)|$ where $\pi$ has only cycles of lengths two or three.
\end{proof}

\begin{theorem}
\label{thm1}
For any $n\geqslant 2r$ and $t\geqslant 2$, we have:
\begin{equation}
N(n,2r,r)=
\begin{cases}
\binom{2t}{t},~~\text{if $r=2t$};\\
2\binom{2t-2}{t-1},~~\text{if $r=2t+1$}.\\
\end{cases}
\label{eq7}
\end{equation}
\end{theorem}
\begin{proof}
By Lemmas~$\ref{lm8}$ and~$\ref{lm9}$, we let $\pi\in \mathrm{Sym}_n$ with $w_H(\pi)=2r$ that is given as a product of disjoint cycles of only lengths two or three such that:
\begin{equation}
I(n,2r,r)=|B_r(\pi)\cap B_r(I_n)|,\nonumber
\end{equation}
where $\pi$ has $p$ cycles of length two and $q$ cycles of length three.

If $r=2t$ then $|Tc(\pi)|=4t$. Thus, $2p+3q=4t$. Let $\sigma \in B_r(\pi)\cap B_r(I_n)$. Since $Tc(\sigma)\subset Tc(\pi)$ and $|Tc(\sigma)|=2t$ then $\sigma$ has $p_1$ cycles of length two and $q_1$ cycles of length three, i.e. $2p_1+3q_1=2t$, where $0\leqslant p_1\leqslant p$ and $0\leqslant q_1\leqslant q$. The number of such permutations $\sigma$ is equal to $\binom{p}{p_1}\binom{q}{q_1}$. Since $2p+3q=4t$ and $p,q,t$ are non-negative integers then $q=2k$ and $p=2t-3k$ for some non-negative integer $k$, where $0\leqslant k\leqslant \frac{2t}{3}$. Since $2p_1+3q_1=2t$ and $p_1,q_1,t$ are non-negative integers then $q_1=2j$ and $p_1=t-3j$ for some non-negative integer $j$, where $0\leqslant j\leqslant k$. So, we have:
\begin{equation}
|B_r(\pi)\cap B_r(I_n)|=\sum\limits_{j=0}^{k}\binom{p}{p_1}\binom{q}{q_1}=\sum\limits_{j=0}^{k}\binom{2t-3k}{t-3j}\binom{2k}{2j}.\nonumber
\end{equation}
Since $\binom{2t-3k}{t-3j}\binom{2k}{2j}\leqslant \binom{2t-k}{t-j}$ for every $j$, $0\leqslant j\leqslant k$, and $\binom{p}{p_1}=\binom{p-1}{p_1}+\binom{p-1}{p_1-1}$ then we have:
\begin{align}
\sum\limits_{j=0}^{k}\binom{2t-3k}{t-3j}\binom{2k}{2j}&\leqslant \sum\limits_{j=0}^{k}\binom{2t-k}{t-j}\nonumber\\
&\leqslant \sum\limits_{j=0}^{k-1}\binom{2t-j-1}{t-j}+\binom{2t-k}{t-k}=\binom{2t}{t}.\nonumber
\end{align}
Thus, we have:
\begin{equation}
|B_r(\pi)\cap B_r(I_n)|\leqslant \binom{2t}{t}.\nonumber
\end{equation}

Let $\pi=(1\,2)(3\,4)\dots(2i-1\,2i)\ldots(4t-1\,4t)$. Then choosing $t$ cycles among $2t$ cycles of $\pi$ one can get a permutation $\sigma$. Hence, we have:
$$|\{\sigma|w_H(\sigma)=2t,d(\sigma,\pi)\leqslant r\}|= \binom{2t}{t},$$ and finally for $r=2t$ we have:
\begin{equation}
I(n,2r,r)=\binom{2t}{t}.\nonumber
\end{equation}

If $r=2t+1$ then $|Tc(\pi)|=4t+2$. Thus, $2p+3q=4t+2$. Let $\sigma \in B_r(\pi)\cap B_r(I_n)$. Since $Tc(\sigma)\subset Tc(\pi)$ and $|Tc(\sigma)|=2t+1$ then $\sigma$ has $p_1$ cycles of length two and $q_1$ cycles of length three such that $2p_1+3q_1=2t+1$, where $0\leqslant p_1\leqslant p$ and $0\leqslant q_1\leqslant q$. There are $\binom{p}{p_1}\binom{q}{q_1}$ such permutations $\sigma$. Since $2p+3q=4t+2$ and $p,q,t$ are non-negative integers then $q=2k$ and $p=2t+1-3k$ for some non-negative integer $k$, $0\leqslant k\leqslant \frac{2t+1}{3}$. Since $2p_1+3q_1=2t+1$ and $p_1,q_1,t$ are non-negative integers then $q_1=1+2j$ and $p_1=t-1-3j$ for some non-negative integer $j$, $0\leqslant j\leqslant k-1$, where $k\geqslant 1$. Therefore, we have: 
$$\binom{p}{p_1}\binom{q}{q_1}=\binom{2t-3k}{t-1-3j}\binom{2k}{1+2j},$$ which gives us:
\begin{equation}
|B_r(\pi)\cap B_r(I_n)|=\sum\limits_{j=0}^{k-1}\binom{2t+1-3k}{t-1-3j}\binom{2k}{1+2j}.\nonumber
\end{equation}

By Lemma~$\ref{lm5}$, we have:
\begin{equation}
|B_r(\pi)\cap B_r(I_n)|\leqslant 2\binom{2t-2}{t-1}.\nonumber
\end{equation}
Moreover, we set $\pi=(1\,2\,4t+1)(3\,4)\ldots(2i-1\,2i)\ldots(4t-1\,4t\,4t+2)$ such that $I(n,2r,r)=|B_r(\pi)\cap B_r(I_n)|=2\binom{2t-2}{t-1}$. Hence, for any $d\geqslant 2r+1$ we have $I(n,d,r)=0$, and finally for $n\geqslant 2r$ we get:
\begin{equation}
N(n,2r,r)=
\begin{cases}
\binom{2t}{t},~~\text{if $r=2t$},\\
2\binom{2t-2}{t-1},~~\text{if $r=2t+1$},\\
\end{cases}
\nonumber
\end{equation}
which completes the proof.
\end{proof}

\section{The lower bound on $N(n,d,r)$ for $d=2r-1$}
\label{sec6}

To give the lower bound on $N(n,2r-1,r)$ we need a few additional results.

\begin{lemma}
\label{lm10}
For any $n\geqslant 2r-1$, there exists a permutation $\pi \in \mathrm{Sym}_n$ with $w_H(\pi)=2r-1$ such that:
\begin{equation}
I(n,2r-1,r)=|B_r(\pi)\cap B_r(I_n)|=\sum\limits_{i=0}^{r}|\{\sigma|w_H(\sigma)=i,d(\sigma,\pi)\leqslant r\}|.\nonumber
\end{equation}
Moreover, if $\sigma\in B_r(\pi)\cap B_r(I_n)$ then $Tc(\sigma)\subset Tc(\pi)$ and $|Tc(\sigma)|=r-1$, or $Tc(\sigma)\subset Tc(\pi)$ and $|Tc(\sigma)|=r$, or $|Tc(\sigma)\cap Tc(\pi)|=r-1$ and $|Tc(\sigma)|=r$.
\end{lemma}
\begin{proof}
By Lemmas~$\ref{lm1}$ and $\ref{lm7}$, there exists a permutation $\pi$ with $w_H(\pi)=2r-1$ such that:
\begin{equation}
I(n,2r,r)=|B_r(\pi)\cap B_r(I_n)|=\sum\limits_{i=0}^{r}|\{\sigma|w_H(\sigma)=i,d(\sigma,\pi)\leqslant r\}|.\nonumber
\end{equation}
Moreover, $d(\sigma,\pi)=|Ts(\sigma)\cup Ts(\pi)|-|Tc(\sigma)\cap Tc(\pi)|\geqslant 2r-1-w_H(\sigma)$. Hence, for any $i$ such that $0\leqslant i\leqslant r-2$ we have $|\{\sigma|w_H(\sigma)=i,d(\sigma,\pi)\leqslant r\}|=0$. The case when $w_H(\sigma)=r-1$ appears if $|Ts(\sigma)\cup Ts(\pi)|=2r-1$ and $|Tc(\sigma)\cap Tc(\pi)|=r-1$. Hence, in this case we have $d(\sigma,\pi)=r$. Furthermore,  $Tc(\sigma)\subset Tc(\pi)$ with $|Tc(\sigma)|=r-1$. Therefore, for each cycle $(x_1\ldots x_l)$ of $\sigma$ there should be a cycle of $\pi$, i.~e., $\sigma$ is presented by a product of disjoint cycles which are cycles of $\pi$.

Similarly, the case $w_H(\sigma)=r$ appears if $|Ts(\sigma)\cup Ts(\pi)|=2r-1$ and $|Tc(\sigma)\cap Tc(\pi)|\leqslant r$. Hence,  $d(\sigma,\pi)\leqslant r$. Furthermore, $Tc(\sigma)\subset Tc(\pi)$ with $|Tc(\sigma)|=r$, or $|Tc(\sigma)\cap Tc(\pi)|=r-1$ with $|Tc(\sigma)|=r$.
\end{proof}

Let $N_1(\pi)$, $N_2(\pi)$ and $N_3(\pi)$ be the sets of $\sigma$ such that 1) $Tc(\sigma)\subset Tc(\pi)$ with $|Tc(\sigma)|=r-1$; 2) $Tc(\sigma)\subset Tc(\pi)$ with $|Tc(\sigma)|=r$; and 3) $|Tc(\sigma)\cap Tc(\pi)|=r-1$, $|Tc(\sigma)|=r$, where  $\pi,\sigma \in \mathrm{Sym}_n$  and $w_H(\pi)=2r-1\leqslant n$. Then the following holds:
$$|B_r(\pi)\cap B_r(I_n)|=|N_1(\pi)|+|N_2(\pi)|+|N_3(\pi)|.$$

\begin{lemma}
\label{lm11}
Let $\pi,\beta \in\mathrm{Sym}_n$  and $w_H(\pi)=w_H(\beta)=2r-1\leqslant n$ for some $r\geqslant 2$. If $Ts(\pi)=Ts(\beta)$ and $\pi$ has all cycles of $\beta$ but a cycle $(y_1\,y_2\ldots y_t)$, also has two cycles $(y_1\,y_2\ldots y_s)$ and $(y_{s+1}\,\ldots y_t)$ with $s\leqslant t-2$, then $|N_1(\pi)|\geqslant |N_1(\beta)|$ and $|N_2(\pi)|\geqslant |N_2(\beta)|$. 
\end{lemma}
\begin{proof}
Let $\sigma\in N_1(\beta)$. Since $Tc(\sigma)\subset Tc(\beta)$ then $\beta$ has all the cycles of $\sigma$. If $\sigma$ does not contain the cycle $(y_1\,y_2\ldots y_t)$ then $\pi$ has all the cycles of $\sigma$ and $\sigma\in N_1(\pi)$. Otherwise, there exists a permutation $\sigma'$ that has all the cycles of $\sigma$ but $(y_1\,\ldots\,y_t)$ with two additional cycles $(y_1\,\ldots\,y_s)$ and $(y_{s+1}\,\ldots\,y_t)$, $|Tc(\sigma')|=r-1$, and $Tc(\sigma')\subset Tc(\pi)$. Thus, $\sigma'\in N_1(\pi)$, and we have $|N_1(\pi)|\geqslant |N_1(\beta)|$. Similarly, we also prove that $|N_2(\pi)|\geqslant |N_2(\beta)|$. 
\end{proof}

\begin{lemma}
\label{lm12}
Let $\pi,\sigma\in\mathrm{Sym}_n$  with $w_H(\pi)=2r-1\leqslant n$ and $w_H(\sigma)=r$ for some $r\geqslant 2$. Then $\sigma\in N_3(\pi)$ if only if $\pi$ has all the cycles of $\sigma$ but a cycle $(y_1\,\ldots\,y_t)$ in $\sigma$, and it has a cycle $(y_1\,\ldots\,y_t\,y_{t+1}\,\ldots\,y_{t+s})$ for some $s\geqslant 1$ and $t\geqslant 2$.
\end{lemma}
\begin{proof}
If $\sigma\in N_3(\pi)$ then $|Tc(\sigma)\cap Tc(\pi)|=r-1$ with $|Tc(\sigma)|=r$. Therefore, there exists only one element $y_t$ in $Tc(\sigma)$ such that $\sigma(y_t)\neq \pi(y_t)$. So, $\pi$ has all the cycles of $\sigma$ but a cycle $(y_1\,\ldots\,y_t)$ in $\sigma$, and also it has another cycle $(y_1\,\ldots\,y_t\,y_{t+1}\,\ldots\,y_{t+s})$ for some $s\geqslant 1$ and $t\geqslant 2$.

On the other hand, if $\pi$ has all the cycles of $\sigma$ but a cycle $(y_1\,\ldots\,y_t)$ in $\sigma$, and also has a cycle $(y_1\,\ldots\,y_t\,y_{t+1}\,\ldots\,y_{t+s})$ for some $s\geqslant 1$ and $t\geqslant 2$, then $|Tc(\sigma)\cap Tc(\pi)|=r-1$ with $|Tc(\sigma)|=r$. Thus, $\sigma\in N_3(\pi)$.
\end{proof}

Now we consider permutations with special properties. 

Let $\mathrm{Sym}_n(2r-1,\ell)$ be a set of permutations $\pi$ with $w_H(\pi)=2r-1$ having a cycle $(y_1\,\ldots\,y_{\ell})$ of the maximum length $\ell\geqslant 3$. We consider a permutation $I^{\ell}\in \mathrm{Sym}_n(2r-1,\ell)$ with nearly all cycles of length two but a cycle of length $\ell$ and a cycle of length three if $\ell$ is even. For example, if $I^4\in \mathrm{Sym}_n(9,4)$ then $disc(I^4)=[1^{n-9}2^13^14^1]$, and if $I^5\in \mathrm{Sym}_n(9,5)$ then $disc(I^5)=[1^{n-9}2^25^1]$. 

In what follows below, let $\binom{0}{0}=1$, $\binom{p}{q}=0$ if $p<q$, $p<0$, or $q<0$.

\begin{lemma}
\label{lm13} For any $n\geqslant 2r-1$ and $t\geqslant 2$, the following statements hold.  
\begin{enumerate}
\item For any $\ell=2s-1$ there exists $I^{\ell} \in \mathrm{Sym}_n(2r-1,\ell)$ with all cycles of length two but a cycle of odd length $\ell$ such that 
    \begin{equation}
    |N_3(I^{\ell})|=
    \begin{cases} 
    \ell\cdot\sum_{i=1}^{s-1} \binom{r-s}{t-i},~$r=2t$;\\
    \ell\cdot\sum_{i=1}^{s-2} \binom{r-s}{t-i},~$r=2t+1$,\\
    \end{cases}\nonumber 
    \end{equation}
    where $2\leqslant s\leqslant r$. 
\item For any $\ell=2s$ there exists $I^{\ell} \in \mathrm{Sym}_n(2r-1,\ell)$ with all cycles of length two but a cycle of length three and a cycle of even length $\ell$ such that
    \begin{equation}
    |N_3(I^{\ell})|=
    \begin{cases} 
    2\cdot(\ell+3)\cdot\binom{r-s-2}{t-1}+\ell\cdot\sum_{i=1}^{s-2}\binom{r-s-1}{t-i-1},~\text{$r=2t$};\\
    \ell\cdot\sum_{i=1}^{s-1}\binom{r-s-1}{t-i} ,~\text{$r=2t+1$};\\
    \end{cases} \nonumber 
    \end{equation}
where $2\leqslant s\leqslant r-2$.
\end{enumerate}
\end{lemma}

\begin{proof} 1. Let $\ell=2s-1$, where $2\leqslant s\leqslant 2t$ and $t\geqslant 2$, and a permutation $I^{\ell} \in \mathrm{Sym}_n(2r-1,\ell)$ has all cycles of length two but a cycle $(y_1\,\ldots\,y_{\ell})$ of odd length $\ell$. If $r=2t$ then $I^{\ell}$ has $r-s$ cycles of length two and a cycle of length $2s-1$, and if $r=2t+1$ then $I^{\ell}$ has $r-s$ cycles of length two and a cycle of length $2s-1$. Assume $\sigma\in N_3(I^{\ell})$ then by Lemma~\ref{lm12} the permutation $\sigma$ has all cycles of length two of $I^{\ell}$ but a cycle $(y_{j+1}\,\ldots\,y_{j+a})$, $1\leqslant j\leqslant {\ell}$, for some $a$ such that $2\leqslant a\leqslant \ell-1=2s-2$. 

If $s=2$ there does not exist $\sigma\in N_3(I^{\ell})$, i.~e. $|N_3(I^{\ell})|=0$ in this case.

If $r=2t$ then since $\sigma\in N_3(I^{\ell})$ we have $w_H(\sigma)=r$ which means $a$ must be even. Let $a=2i$ and $1\leqslant i\leqslant s-1$. It follows  $\sigma$ has $t-i$ cycles of length two from $r-s$ cycles of length two of $I^{\ell}$ and a cycle of length $2i$ among cycles of type $(y_1\,\ldots\,y_{\ell})$ of $I^{\ell}$. By Lemma~\ref{lm12}, the cycle of length $2i$ is a continuous substring of $(y_1\,\ldots\,y_{\ell})$, and hence in this case the number of permutations $\sigma\in N_3(I^{\ell})$ under consideration is $\ell\cdot\binom{r-s}{t-i}$ for some $1\leqslant i\leqslant s-1$. 

Let $r=2t+1$ then $a$ must be odd with taking values $3,5,\ldots,2s-3$. Let $a=2i+1$ and $1\leqslant i\leqslant s-2$. It follows that $\sigma$ has $t-i$ cycles of length two from $r-s$ cycles of length two of $I^{\ell}$ and a cycle of length $2i+1$ among cycles of type $(y_1\,\ldots\,y_{\ell})$ of $I^{\ell}$. By Lemma~\ref{lm12}, the cycle of length $2i+1$ is a continuous substring of $(y_1\,\ldots\,y_{\ell})$, and hence in this case the total number of sought permutations $\sigma\in N_3(I^{\ell})$ is $\ell\cdot\binom{r-s}{t-i}$ for some $1\leqslant i\leqslant s-2$. 

2. Let $\ell=2s$, where $s\geqslant 2$, and a permutation $I^{\ell} \in \mathrm{Sym}_n(2r-1,\ell)$ has all cycles of length two but a cycle $(z_1\,z_2\,z_3)$ of length three and a cycle $(y_1\,\ldots\,y_{\ell})$ of even length $\ell$. For any $t\geqslant 2$, if $r=2t$ then $I^{\ell}$ has $r-s-2$ cycles of length two, one cycle of length three and one cycle of length $2s$, and if $r=2t+1$ then $I^{\ell}$ has $r-s-2$ cycles of length two, one cycle of length three and one cycle of length $2s$. 

Assume $\sigma\in N_3(I^{\ell})$. By Lemma~\ref{lm12}, there are the following four scenarios:
\begin{enumerate}
    \item $\sigma$ has all cycles of length two of $I^{\ell}$ but a cycle $(y_1\,\ldots\,y_a)$ for some $a$, where $2\leqslant a\leqslant \ell-1=2s-1$;
    \item $\sigma$ has all cycles of length two of $I^{\ell}$ but a cycle $(z_1\,z_2\,z_3)$ and a cycle $(y_1\,\ldots\,y_a)$ for some $a$, where $2\leqslant a\leqslant 2s-1$;
    \item $\sigma$ has all cycles of length two of $I^{\ell}$ but a cycle $(z_{a}\,z_{a+1})$ for some $a$, where $1\leqslant a\leqslant 3$; 
    \item $\sigma$ has all cycles of length two of $I^{\ell}$ but a cycle $(y_1\,y_2\,\ldots\,y_{\ell})$ and a cycle $(z_{a}\,z_{a+1})$ for some $a$, $1\leqslant a\leqslant 3$. 
\end{enumerate}

If $r=2t$ then since $\sigma\in N_3(I^{\ell})$ we have $w_H(\sigma)=r$. By the first scenario, $\sigma$ has $t-i$ cycles of length two taken from $r-s-2$ cycles of length two of $I^{\ell}$ and a cycle of length $a=2i$ obtained from $(y_1\,\ldots\,y_{\ell})$ in $I^{\ell}$, $1\leqslant i\leqslant s-1$. So, in this case the number of permutations $\sigma\in N_3(I^{\ell})$ in the first scenario is given by $\ell\cdot\binom{r-s-2}{t-i}$, where $1\leqslant i\leqslant s-1$. Furthermore, the number of desired permutations in the second scenario is given by $\ell\cdot\binom{r-s-2}{t-i-2}$, where $1\leqslant i\leqslant s-1$, and the corresponding numbers of permutations in the third and the fourth scenarios are given by $3\cdot\binom{r-s-2}{t-1}$ and $3\cdot\binom{r-s-2}{t-s-1}$, respectively, which finally gives us the formula for $|N_3(I^{\ell})|$ in the case $r=2t$ and $2\leqslant s\leqslant r-2$. 

If $r=2t+1$ then the formula for $|N_3(I^{\ell})|$ is obtained by similar arguments.
\end{proof}

Lemmas~\ref{lm11} and~\ref{lm13} give the following result.

\begin{lemma}
\label{lm14}  For any $n\geqslant 2r-1$ and $t\geqslant 2$, we have:
\begin{align*}
|B_r&(I^{\ell})\cap B_r(I_n)|=|N_1(I^{\ell})|+|N_2(I^{\ell})|+|N_3(I^{\ell})|=\\
&=
\begin{cases} 
2\cdot\binom{r-s}{t}+\ell\cdot\sum_{i=1}^{s-1} \binom{r-s}{t-i},~\text{$r=2t$};\\
2\cdot\binom{r-s}{t}+\ell\cdot\sum_{i=1}^{s-2} \binom{r-s}{t-i},~\text{$r=2t+1$},\\
\end{cases}
\end{align*}

where $\ell=2s-1$ and $2\leqslant s\leqslant r$;
\begin{align*}
&=
\begin{cases} 
2\cdot\binom{r-s}{t}+2\cdot(\ell+1)\cdot\binom{r-s-2}{t-1}+\ell\cdot\sum_{i=1}^{s-2}\binom{r-s-1}{t-i-1},~\text{$r=2t$},\\
2\cdot\binom{r-s-1}{t}+\ell\cdot\sum_{i=1}^{s-1}\binom{r-s-1}{t-i},~\text{$r=2t+1$},\\
\end{cases}
\end{align*}

where $\ell=2s$ and $2\leqslant s\leqslant r-2$.
\end{lemma}
\begin{proof} If $\ell=2s-1$ and $r=2t$ then $I^{\ell}$ has $r-s$ cycles of length two and a cycle of length $2s-1$. Let $\sigma\in N_1(I^{\ell})$ then $w_H(\sigma)=r-1$. Thus, $\sigma$ is presented by $t-s$ cycles of length two which are in the list of $r-s$ cycles of length two of $I^{\ell}$, and it also has a cycle of length $2s-1$ from $I^{\ell}$. Hence, we have: 
\begin{equation}
|N_1(I^{\ell})|=\binom{r-s}{t-s}.\nonumber
\end{equation}
Now let $\sigma\in N_2(I^{\ell})$ with $w_H(\sigma)=r=2t$ then $\sigma$ has $t$ cycles of length two from the list of those cycles in $I^{\ell}$ which immediately gives us: 
\begin{equation}
|N_2(I^{\ell})|=\binom{r-s}{t}.\nonumber
\end{equation}
Therefore, by Lemma~\ref{lm13} we get the resulting formula in the case when $\ell=2s-1$ and $r=2t$. 

The other cases of $\ell$ and $r$ are obtained by a similar way.
\end{proof}

Now we are ready to prove the following statement. 

\begin{lemma}
\label{lm15}  For any $n\geqslant 2r-1$ and $t\geqslant 2$, we have:
\begin{equation}\label{eq8}
\max\limits_{3\leqslant \ell\leqslant 2r-1} |B_r(I^{\ell})\cap B_r(I_n)|=2\cdot\binom{2t-3}{t}+ \ell^*\cdot\binom{2t-2}{t-1},
\end{equation} 
where the maximum is reached for $\ell^*=5$ if $r=2t$ and for $\ell^*=7$ if $r=2t+1$.
\end{lemma}

\begin{proof}
Let $r=2t$. If $\ell=2s-1$ then by Lemma~\ref{lm14} we have:  
\begin{equation}\label{eq9}
|B_r(I^{\ell})\cap B_r(I_n)|=2\cdot\binom{2t-s}{t}+\ell\cdot\sum_{i=1}^{s-1} \binom{2t-s}{t-i}
\end{equation} 
for any $2\leqslant s\leqslant 2t$. If $\ell=3$ and $\ell=5$ then~(\ref{eq9}) gives $2\binom{2t-2}{t}+3\binom{2t-2}{t-1}$ and $2\binom{2t-3}{t}+5\left(\binom{2t-3}{t-1}+\binom{2t-3}{t-2}\right)$, respectively. Since
\begin{equation}\label{eq10}
\binom{n}{k}=\binom{n-1}{k}+\binom{n-1}{k-1}
\end{equation}
we have: 
\begin{align*}
2\binom{2t-2}{t}+3\binom{2t-2}{t-1}&=2\binom{2t-3}{t}+2\binom{2t-3}{t-1}+3\binom{2t-3}{t-1}+3\binom{2t-3}{t-2}\\
&\leqslant 2\binom{2t-3}{t}+5\binom{2t-3}{t-1}+5\binom{2t-3}{t-2}.
\end{align*}
For any $s\geqslant 3$, it is easily shown the following equation: 
$$|B_r(I^{2s-1})\cap B_r(I_n)|-|B_r(I^{2s+1})\cap B_r(I_n)|=$$
$$=(2s-3)\cdot\sum_{i=2}^{s-1} \binom{2t-s-1}{t-i}-2\cdot\binom{2t-s-1}{t-1}.$$ 

Since $s\geqslant 3$ then 
$$\binom{2t-s-1}{t-2}\geqslant \binom{2t-s-1}{t-1}$$
as $(2t-s-1)/2\leqslant (t-2)$. Therefore, we have:
$$|B_r(I^{2s-1})\cap B_r(I_n)|\geqslant |B_r(I^{2s+1})\cap B_r(I_n)|$$ 
and finally we get: 
\begin{equation}
|B_r(I^5)\cap B_r(I_n)|=\max\limits_{3\leqslant \ell\leqslant 2t-1}|B_r(I^{\ell})\cap B_r(I_n)|,\nonumber    
\end{equation}
when $r=2t$ and $\ell=2s-1$.

If $\ell=2s$ then by Lemma~\ref{lm14} the following holds:  
$$|B_r(I^{\ell})\cap B_r(I_n)|=2\cdot\binom{2t-s}{t}+2\cdot(\ell+1)\cdot\binom{2t-s-2}{t-1}+\ell\cdot\sum_{i=1}^{s-2}\binom{2t-s-1}{t-i-1}$$
for any $2\leqslant s\leqslant 2t-2$. 

Let us compare $|B_r(I^{2s})\cap B_r(I_n)|$ and $|B_r(I^{2s+1})\cap B_r(I_n)|$ as follows. By Lemma~\ref{lm14}, we have:
$$|B_r(I^{2s+1})\cap B_r(I_n)|=2\cdot\binom{2t-s-1}{t}+(\ell+1)\cdot\sum_{i=1}^s\binom{2t-s-1}{t-i},$$
and for any $s\geqslant 2$ we get: 
$$|B_r(I^{2s+1})\cap B_r(I_n)|-|B_r(I^{2s})\cap B_r(I_n)|=$$
$$=2\ell\cdot\binom{2t-s-2}{t-2}+\sum_{i=2}^{s-1}\binom{2t-s-1}{t-i}-2\cdot\binom{2t-s-2}{t-1}.$$
Since $s\geqslant 2$ then $(2t-s-2)/2\leqslant (t-2)$ which gives $\binom{2t-s-2}{t-2}\geqslant \binom{2t-s-2}{t-1}$. It follows that we have the following inequality: 
$$|B_r(I^{2s+1})\cap B_r(I_n)|\geqslant |B_r(I^{2s})\cap B_r(I_n)|$$  
and finally we get: 
\begin{equation}
|B_r(I^5)\cap B_r(I_n)|=\max\limits_{3\leqslant \ell\leqslant 2t-1}|B_r(I^{\ell})\cap B_r(I_n)|,\nonumber    
\end{equation}
when $r=2t$ and $\ell=2s$.

The case $r=2t+1$ is proved using similar arguments. If $\ell=2s-1$ then by Lemma~\ref{lm14} we have: 
\begin{equation} \label{eq11}
|B_r(I^l)\cap B_r(I_n)|= 2\cdot\binom{2t+1-s}{t}+\ell\cdot\sum_{i=1}^{s-2} \binom{2t+1-s}{t-i}
\end{equation}
for any $2\leqslant s\leqslant 2t+1$. If $\ell=3$ or $\ell=5$ then~(\ref{eq11}) gives $2\binom{2t-1}{t}$ or $2\binom{2t-2}{t}+5\binom{2t-2}{t-1}$, respectively, and we have: 
\begin{align*}
2\binom{2t-1}{t}=2\binom{2t-2}{t}+2\binom{2t-2}{t-1}\leqslant 2\binom{2t-2}{t}+5\binom{2t-2}{t-1}.
\end{align*}
Moreover, if $\ell=7$ then~(\ref{eq11}) is equal to $2\binom{2t-3}{t}+7\binom{2t-2}{t-1}$ and we have the following inequality:
$$2\binom{2t-2}{t}+5\binom{2t-2}{t-1}\leqslant 2\binom{2t-3}{t}+7\binom{2t-2}{t-1}.$$
For any $s\geqslant 4$, we get:
$$|B_r(I^{2s-1})\cap B_r(I_n)|-|B_r(I^{2s+1})\cap B_r(I_n)|=$$
$$=(\ell-2)\cdot\sum_{i=2}^{s-1} \binom{2t-s}{t-i}-2\cdot\binom{2t-s}{t-1},$$ 
and since $(2t-s)/2\leqslant (t-2)$ we have $\binom{2t-s}{t-2}\geqslant \binom{2t-s}{t-1}$. This means that 
$$|B_r(I^{2s-1})\cap B_r(I_n)|\geqslant |B_r(I^{2s+1})\cap B_r(I_n)|$$ for any $s\geqslant 4$, and we have: 
\begin{equation}
|B_r(I^7)\cap B_r(I_n)|=\max\limits_{3\leqslant \ell\leqslant 2t-1}|B_r(I^{\ell})\cap B_r(I_n)|,\nonumber    
\end{equation}
when $r=2t+1$ and $\ell=2s-1$.

If $\ell=2s$ then by Lemma~\ref{lm14} we have: 
$$|B_r(I^{\ell})\cap B_r(I_n)|=2\cdot\binom{2t-s}{t}+\ell\cdot\sum_{i=1}^{s-1}\binom{2t-s}{t-i}$$
for any $2\leqslant s\leqslant 2t-1$. 

Now we compare $|B_r(I^{2s})\cap B_r(I_n)|$ and $|B_r(I^{2s+1})\cap B_r(I_n)|$ as follows. By~(\ref{eq10}) and Lemma~\ref{lm14}, we have:
$$|B_r(I^{2s+1})\cap B_r(I_n)|=2\cdot\binom{2t-s}{t}+(\ell+1)\cdot\sum_{i=1}^{s-1}\binom{2t-s}{t-i},$$
and for any $s\geqslant 2$ we get: 
$$|B_r(I^{2s+1})\cap B_r(I_n)|-|B_r(I^{2s})\cap B_r(I_n)|=\sum_{i=1}^{s-1}\binom{2t-s}{t-i}.$$ 
Thus, it follows that 
$$|B_r(I^{2s+1})\cap B_r(I_n)|\geqslant |B_r(I^{2s})\cap B_r(I_n)|$$ 
and finally we have: 
\begin{equation}
|B_r(I^7)\cap B_r(I_n)|=\max\limits_{3\leqslant \ell\leqslant 2t-1}|B_r(I^{\ell})\cap B_r(I_n)|,\nonumber    
\end{equation}
when $r=2t+1$ and $\ell=2s$, which completes the proof.
\end{proof}

Let us note that if $t=2$ then by Lemma~\ref{lm15} we have $|B_4(I^5)\cap B_4(I_n)|=10$, where $I^5=(1\,2)(3\,4\,5\,6\,7)$, and we have $|B_5(I^7)\cap B_5(I_n)|=14$, where $I^7=(1\,2)(3\,4\,5\,6\,7\,8\,9)$. In what follows below we give two examples showing that $N(n,7,4)=|B_4(I^5)\cap B_4(I_n)|$ and $N(n,9,5)>|B_5(I^7)\cap B_5(I_n)|$. 

\begin{example}
Let $d=7$ and $\pi\in \mathrm{Sym}_n$ with $w_H(\pi)=7$ then $disc(\pi)\in \{[7^1],[5^12^1],[4^13^1],[3^12^2]$\}. By Lemmas~\ref{lm7} and~\ref{lm12}, we have $|B_4(\pi)\cap B_4(I_n)|\in\{7,10,2,8\}$ (in corresponding order). Thus, $I(n,7,4)=10$. By Theorem~\ref{thm1}, we have $I(n,8,4)=\binom{4}{2}=6$, hence, $N(n,7,4)=|B_4(I^5)\cap B_4(I_n)|=10$.    
\end{example}

\begin{example}
Let $d=9$ and $\pi\in \mathrm{Sym}_n$ with $w_H(\pi)=9$ then $$disc(\pi)\in\{[9^1],[7^12^1],[6^13^1],[5^14^1],[5^12^2],[4^13^12^1],[3^3],[3^12^3]\}$$
and
$|B_5(\pi)\cap B_5(I_n)|\in\{9,14,12,2,12,10,18,6\}$  (in corresponding order). Therefore, $I(n,9,5)=18$. By Theorem~\ref{thm1}, $I(n,10,5)=2\binom{2}{1}=4$, hence, $N(n,9,5)=18$. Thus, the value $|B_5(I^7)\cap B_5(I_n)|=14$ is suboptimal.
\end{example}

In the same way it is easily to check that $N(n,5,3)=5$ and $N(n,3,2)=3$. 

Let us use the discussion above to prove the lower bound on $N(n,2r-1,r)$ for any $n\geqslant 2r-1$.

\begin{theorem}
\label{thm2}
For any $n\geqslant 2r-1$ and $t\geqslant 2$, we have:
\begin{equation}\label{eq12}
I(n,2r-1,r)\geqslant 2\cdot\binom{2t-3}{t}+ \ell^*\cdot\binom{2t-2}{t-1},
\end{equation} 
where $\ell^*=5$ if $r=2t$ and $\ell^*=7$ if $r=2t+1$.
Moreover, for any $n\geqslant 2r$, we have:
\begin{equation} \label{eq13}
N(n,2r-1,r)=I(n,2r-1,r)\geqslant N(n,2r,r).
\end{equation}
Furthermore, we have
\begin{equation} \label{eq14}
N(n,2r-1,r)\geqslant 2\cdot\binom{2t-3}{t}+ \ell^*\cdot\binom{2t-2}{t-1}
\end{equation} 
over the same conditions for $\ell^*$. 
\end{theorem}
\begin{proof}
By Lemma~\ref{lm15}, if $r=2t$ then we have:  
\begin{equation}
I(n,2r-1,r)\geqslant |B_r(I^5)\cap B_r(I_n)|=2\binom{2t-3}{t}+5\binom{2t-2}{t-1},\nonumber
\end{equation}
 where $I^5=(1\,2)(3\,4\,5\,6\,7)$, and if $r=2t+1$ then we get:
\begin{equation}
I(n,2r-1,r)\geqslant |B_r(I^7)\cap B_r(I_n)|=2\binom{2t-3}{t}+7\binom{2t-2}{t-1},
\nonumber
\end{equation}
where $I^7=(1\,2)(3\,4\,5\,6\,7\,8\,9)$, that all together lead to~(\ref{eq12}). 

Now let us prove~(\ref{eq13}) and~(\ref{eq14}). By Theorem~\ref{thm1}, if $r=2t$ then $N(n,2r,r)=I(n,2r,r)=\binom{2t}{t}$ such that: 
\begin{align*}
I(n,2r,r)=\binom{2t}{t}&=2\binom{2t-3}{t}+2\binom{2t-3}{t-1}+2\binom{2t-2}{t-1}\\
&\leqslant 2\binom{2t-3}{t}+5\binom{2t-2}{t-1}\leqslant  I(n,2r-1,r), 
\end{align*}
and if $r=2t+1$ then $N(n,2r,r)=I(n,2r,r)=2\binom{2t-2}{t-1}$ where:
\begin{align*}
  I(n,2r,r)=2\binom{2t-2}{t-1}< 2\binom{2t-3}{t}+7\binom{2t-2}{t-1}\leqslant I(n,2r-1,r),
\end{align*}
which gives~(\ref{eq13}). Moreover, since $I(n,d,r)=0$ for any $d\geqslant 2r+1$ then $I(n,2r-1,r)\geqslant I(n,d,r)$ for any $d\geqslant 2r$. Thus, $ N(n,2r-1,r)=I(n,2r-1,r)\geqslant N(n,2r,r)$ which gives~(\ref{eq14}) for any $t\geqslant 2$ and completes the proof. 
\end{proof}

\section{Further discussions}

It is not difficult to show that the lower bound in~(\ref{eq14}) is attained for $r=6,7,8,9$. Indeed, by Lemma~\ref{lm7} and Theorem~\ref{thm2} we have:
$$N(n,11,6)=2\binom{3}{3}+5\binom{4}{2}=32, \hspace{5mm} N(n,13,7)=2\binom{3}{3}+7\binom{4}{2}=44,$$ 
$$N(n,15,8)=2\binom{5}{4}+5\binom{6}{3}=110, \hspace{5mm} N(n,17,9)=2\binom{5}{4}+7\binom{6}{3}=150.$$

The following natural question arises here. 

\begin{open} Is it true that for any $t\geqslant 5$ the following holds:
\begin{equation*}
N(n,2r-1,r)=2\cdot\binom{2t-3}{t}+ \ell^*\cdot\binom{2t-2}{t-1},
\end{equation*} 
where $\ell^*=5$ if $r=2t$ and $\ell^*=7$ if $r=2t+1$.
\end{open}

\section*{Acknowledgements}
The work of Xiang Wang is supported by the National Natural Science Foundation of China (Grant No. 12001134), and the work of Elena~V.~Konstantinova is supported by the Mathematical Center in Akademgorodok, under agreement No. 075-15-2022-281 with the Ministry of Science and High Education of the Russian Federation.


\begin{thebibliography}{}

\bibitem{CKLL97}
G.~Cohen, I.~Honkala, S.~Litsyn, A,~Lobstein, \emph{Covering Codes}, Ser. North-Holland Mathematical Library, Vol.~54. Amsterdam, The Netherlands: North-Holland, pp.~16-17, 1997.

\bibitem{KG90}
M.~Kendall, J.~D.~Gibbons, \emph{Rank Correlation Methods}. New York: Oxford Univ. Press, 1990.

\bibitem{K1}
E.~Konstantinova, Reconstruction of permutations distorted by single reversal errors, \emph{Discrete Applied Math.}, \textbf{155} (2007) 2426--2434.\par

\bibitem{K2}
E.~Konstantinova, V.~Levenshtein, and J.~Siemons, Reconstruction of permutations distorted by single transposition errors, 2007, \texttt{http://arxiv.org/abs/math/0702191v1}.\par

\bibitem{K3}
E.~Konstantinova, On reconstruction of signed permutations distorted by reversal errors, \emph{Discrete Mathematics}, \textbf{308} (2008) 974--984.\par

\bibitem{K4}
E.~V.~Konstantinova, Vertex reconstruction in Cayley graphs, \emph{Discrete Mathematics}, \textbf{309} (2009) 548--559. 

\bibitem{Knu94}
D.~E.~Knuth, \emph{Sorting and Searching}, Vol. 3 of The Art of Computer Programming, Addison--Wesley, Reading, Massachusetts, second
edition, 1998.

\bibitem{L1}
V.~I.~Levenshtein, Efficient reconstruction of sequences, \emph{IEEE Trans. on Inform. Theory}, \textbf{47}(1) (2001) 2--22.\par

\bibitem{L2}
V.~I.~Levenshtein, Efficient reconstruction of sequences from their subsequences or supersequences, \emph{Journal of Combin. Theory, Ser. A}, \textbf{93}(2) (2001) 310--332.\par

\bibitem{L4}
V.~I.~Levenshtein, J.~Siemons, Error graphs and the reconstruction of elements in groups, \emph{Journal of Combin. Theory, Ser. A}, \textbf{116} (2009) 795--815.\par

\bibitem{Stanley11}
R.~Stanley, \emph{Enumerative Combinatorics}, 2nd ed., Cambridge Studies in Advanced Mathematics, Cambridge: Cambridge University Press, 2011.

\bibitem{Yaakobi}
E.~Yaakobi, M.~Schwartz, M.~Langberg, and J.~Bruck, Sequence reconstruction for Grassmann graphs and permutations, In: \emph{Proc. Int. Symp. Inform. Theory}, 2013, 874--878.\par

\bibitem{Wang1}
X.~Wang, Reconstruction of permutations distorted by single Kendall $\tau$-errors, \emph{Cryptogr. Commun.}, \textbf{15} (2023) 131--144. 

\end{thebibliography}
\end{document}